\documentclass{amsart}

\usepackage{latexsym}
\usepackage{amssymb}
\usepackage{amsfonts}
\usepackage{amstext}
\usepackage{hyperref}
\usepackage{amsmath}
\usepackage[all]{xy}
\usepackage{graphicx}

\newcommand{\R}{\mathbb{R}}
\newcommand{\C}{\mathbb{C}}
\newcommand{\Z}{\mathbb{Z}}
\newcommand{\Q}{\mathbb{Q}}
\newcommand{\N}{\mathbb{N}}
\newcommand{\K}{\mathbb{K}}
\renewcommand{\P}{\mathbb{P}}

\newcommand{\OO}{\mathcal{O}}
\newcommand{\CC}{\mathcal{C}}

\newcommand{\EE}{\mathcal{E}}

\newcommand{\HH}{\mathcal{H}}

\newcommand{\mmm}{\mathfrak{m}}

\DeclareMathOperator{\ord}{ord}

\DeclareMathOperator{\Jac}{Jac}

\DeclareMathOperator{\divi}{div}

\theoremstyle{plain}
\newtheorem{thm}{Theorem}[section]
\newtheorem{dfn}[thm]{Definition}
\newtheorem{lem}[thm]{Lemma}
\newtheorem{pro}[thm]{Proposition}
\newtheorem{cor}[thm]{Corollary}

\newtheorem*{thmA}{Theorem A}

\newtheorem*{thmC}{Theorem C}

\newtheorem*{corB}{Corollary B}

\newtheorem*{corD}{Corollary D}

\newcommand{\Xreg}{X_{\text{reg}}}

\DeclareMathOperator{\Exc}{Exc}

\newcommand{\Mp}[1]{}
\newcommand{\Mpe}[1]{}

\newcommand{\DS}{\text{DS}}

\numberwithin{equation}{section}

\title{The Jacobian cocycle and equidistribution towards the Green current}
\author{Rodrigo Parra}
\address{Department of Mathematcs, University of Michigan, 530 Church Street, Ann Arbor, MI 48109-1043.}
\email{rparra@umich.edu}
\begin{document}
\maketitle

\begin{abstract}
Given a holomorphic self-map of complex projective space of degree larger than one, we prove that there exists a finite collection of totally invariant algebraic sets with the following property: given any positive closed (1,1)-current of mass 1 with no mass on any element of this family, the sequence of normalized pull-backs of the current converges to the Green current. Under suitable geometric conditions on the collection of totally invariant algebraic sets, we prove a sharper equidistribution result.
\end{abstract}

\tableofcontents

\section{Introduction}

Let $\P^k$ be the complex projective space of dimension $k$ and let $f:\P^k\to\P^k$ be a holomorphic map of algebraic degree $d\geq2$. It is well known (see \cite{MR1285389}, \cite{HubbardPapadopol}) that there exists a positive closed (1,1)-current $T_f$, the \emph{Green current}, such that for every smooth (1,1)-form $\alpha$ in the cohomology class of the Fubini-Study form $\omega$, the sequence of smooth (1,1)-forms $d^{-n}(f^n)^*\alpha$ converges to $T_f$ in the sense of currents. A natural question to ask is if such behavior also occurs when we replace smooth forms by currents. More precisely, if $S$ is a positive closed (1,1)-current of mass 1, when does the convergence

\begin{equation}\label{equidistribution}
d^{-n}(f^n)^*S\rightarrow T_f
\end{equation}	\Mpe{equidistribution} hold? This last convergence is what we refer to as \emph{equidistribution}. The answer: Not always. Assume for example that there exists a \emph{totally invariant} irreducible hypersurface $X\subset\P^k$, i.e. $f^{-s}(X)=X$ for some $s\in\N$ (which for simplicity we take to be $s=1$); then its current of integration $[X]$ satisfies $f^*[X]=d[X]$ giving us that

\begin{equation*}
d^{-n}(f^n)^*[X]=[X]\nrightarrow T_f
\end{equation*} since the current $T_f$ has no mass on any algebraic subsets of $\P^k$ (in particular, it cannot be the current of integration of an algebraic variety. See \cite{SibSurvey} for more details). Therefore, it seems that the appearance of totally invariant algebraic sets restricts the possibility of having equidistribution. A very important feature of holomorphic maps is that the collection of all totally invariant algebraic subsets of $\P^k$ is finite (see \cite{MR1285389}, \cite{MR2468484}).\\

In dimension $k=1$ (i.e. $\P^1$ is the Riemann sphere and $T_f$ is an invariant probability measure), a famous result by Brolin \cite{MR0194595} (for the case of a polynomial self-maps of $\C$) and by Lyubich \cite{MR741393}, Freire-Lopes-Ma\~n\'e \cite{MR736568} (for the case of rational self-maps of $\P^1$) states that there exists a collection $\EE_f$ of totally invariant points (also called \emph{exceptional points}), of cardinality at most 2, with the following property:\\

Given any probability measure $\nu$ on $\P^1$, $d^{-n}(f^n)^*\nu$ converges to $T_f$ if and only if $\nu(\{p\})=0$ for all $\{p\}\in\EE_f$. In particular, for every $x\in\P^1$ which is not exceptional

\begin{equation}\label{equidistributionpreimages}
d^{-n}(f^n)^*\delta_x=\frac1{d^n}\sum_{f^n(y)=x}\delta_y\to T_f
\end{equation}\Mpe{equidistributionpreimages} as $n\to+\infty$, where $\delta_x$ denotes the Dirac mass at $x$. The equation \eqref{equidistributionpreimages} also shows that the sequence of preimages of points outside $\EE_f$ accumulate along the Julia set of $f$.\\

The situation for $k=2$ is already highly more involved. Some partial results for equidistribution in $\P^2$ for holomorphic (and meromorphic) maps were obtained by J. E. Forn{\ae}ss and N. Sibony \cite{MR1369137}, A. Russakovskii. and B. Shiffman \cite{RussakovskiiShiffman-Sequences} and others. C. Favre and M. Jonsson finished the characterization for the two-dimensional case in \cite{MR2032940} (see also \cite{MR2339287}) proving the following: There exists a family $\EE_f$ of totally invariant irreducible algebraic subsets of $\P^2$, containing at most 3 lines and a finite number of points, with the following property: given any positive closed (1,1)-current $S$ of mass 1, $d^{-n}(f^n)^*S$ converges to $T_f$ if and only if $S$ has no mass on any element of $\EE_f$. The elements of $\EE_f$ are attracting in nature and this collection can be strictly smaller than the collection of all totally invariant irreducible algebraic subsets of $\P^k$.\\

In higher dimensions the situation is not as well understood, particularly since we do not have any satisfactory classification of totally invariant algebraic subsets of $\P^k$. In order to tackle this, we make an assumption on the singular locus of totally invariant algebraic subsets which will allow us to develop our methods.\\

We present our main result

\begin{thmA}\label{theoremA}
Let $f:\P^k\to\P^k$ be a holomorphic map of algebraic degree $d\geq2$ and assume that all totally invariant algebraic subsets have normalizations with at worst isolated quotient singularities. Then there exists a finite collection $\EE_f$ of irreducible totally invariant algebraic sets with the following property: given any positive closed (1,1)-current $S$ of mass 1 with no mass on any element of $\EE_f$ we have

$$d^{-n}(f^n)^*S\rightarrow T_f$$ as $n\to+\infty$ in the sense of currents.
\end{thmA}

The finite family $\EE_f$ coincides with the ones already obtained for $k=1$ and $k=2$. It is constructed following the ideas of Favre and Jonsson in \cite{MR2032940}, but we need to push the methods further and overcome technical difficulties that appear in dimension three and higher.\\

Our assumption on the singularities of the totally invariant algebraic subsets (namely, to have normalizations with at worst isolated quotient singularities) holds rather trivially in dimension 2, and also in dimension 3, as can be derived from the work of J. Wahl \cite{MR1044058}, N. Nakayama \cite{NakayamaRIMS}, D.Q. Zhang \cite{Zhang-classification} and C. Favre \cite{FavreSelfMapsSurfaces} (see Section 6 for details). From this we obtain as a corollary a sharper equidistribution result in dimension 3

\begin{corB}\label{corollaryB}
Let $f:\P^3\to\P^3$ be a holomorphic map of degree $d\geq2$. There exists a finite collection $\EE_f$ of irreducible totally invariant algebraic sets with the following property: given any positive closed (1,1)-current $S$ with no mass on any element of $\EE_f$, the sequence $d^{-n}(f^n)^*S$ converges to $T_f$ in the sense of currents.
\end{corB}

We conjecture that the converse implication is also true: if the sequence $d^{-n}(f^n)^*S$ converges to $T_f$ then $S$ has no mass on any element of $\EE_f$; this would extend the results already known in dimensions one and two.\\

The equidistribution problem in higher dimensions was studied already in \cite{MR1369137}, \cite{RussakovskiiShiffman-Sequences} and \cite{SibSurvey}. In \cite{MR2017143} V. Guedj showed that for a given positive closed (1,1)-current $S$ with Lelong numbers zero everywhere we have $d^{-n}(f^n)^*S\to T_f$ as $n\to+\infty$ (his result also holds for $f$ meromorphic). In \cite{MR2468484} T.-C. Dinh and N. Sibony established the following: There exists a finite collection $\EE_\DS$ of totally invariant irreducible algebraic subsets of $\P^k$ with the following property: given any positive closed (1,1)-current $S$ of mass 1 whose local potentials are not identically $-\infty$ on any element of $\EE_\DS$, $d^{-n}(f^n)^*S$ converges to $T_f$ (their result is in fact uniform in $S$ in a certain sense). The collection $\EE_\DS$ obtained by Dinh and Sibony is constructed inductively by studying the induced dynamics on totally invariant sets. Note that neither Guedj's result nor Dinh and Sibony's result imply each other.\\

It is important to notice that the families $\EE_f$ and $\EE_\DS$ are different and they satisfy the relation

$$\EE_f\subset\EE_\DS\subset\left\{\text{All totally invariant alg subsets of } \P^k\right\}$$ where the first and/or the second inclusion can be strict\\

Our techniques will also allow us to generalize what has been so far obtained by Dinh-Sibony and Guedj and provide the sharpest results known by the author for dimensions $k\geq3$ in the holomorphic setting. We extend Dinh-Sibony's result

\begin{thmC}\label{theoremC}
Let $f:\P^k\to\P^k$ be a holomorphic map of degree $d\geq2$. There exists a finite collection $\EE_\DS$ of totally invariant algebraic subsets of $\P^k$ with the following property: given any positive closed (1,1)-current $S$ with no mass on any element of $\EE_\DS$ we have

$$d^{-n}(f^n)^*S\rightarrow T_f$$ as $n\to+\infty$ in the sense of currents.
\end{thmC}

As opposed to Theorem A, we make no assumptions on the singularities of the totally invariant algebraic subsets. The collection $\EE_\DS$ is the same one constructed by Dinh and Sibony in \cite{MR2468484} (see also \cite{MR2513537} for a different construction) and our improvement is based on approximating the current $S$ by currents that satisfy the conditions imposed by Dinh and Sibony in their theorem.\\

As an immediate consequence of Theorem C, we can extend Guedj's result

\begin{corD}\label{corD}
Let $f:\P^k\to\P^k$ be a holomorphic map of degree $d\geq2$. There exist a totally invariant proper algebraic subset $E\subset\P^k$ with the following property: given any positive closed (1,1)-current $S$ such that $\nu(S,x)=0$ for all $x\in E$ we have
$$d^{-n}(f^n)^*S\rightarrow T_f$$ as $n\to+\infty$ in the sense of currents, where $\nu(S,x)$ denotes the Lelong number of $S$ at $x$.
\end{corD}

The proof of Corollary D follows by taking $E$ as the union of the elements of $\EE_\DS$ (which is algebraic). On the other hand, the tools introduced in this paper will allow us to present a direct proof of Corollary D.\\

In order to verify equidistribution, we will use a characterization due to V. Guedj (see \cite{MR2017143}, Theorem 1.4) which states that it is enough to test the asymptotic behavior of Lelong numbers. More precisely, he proved the equivalence

\begin{equation}\label{Guedjcharact}
d^{-n}(f^n)^*S\to T_f\Longleftrightarrow \sup_{x\in\P^k}\nu\left(d^{-n}(f^n)^*S,x\right)\to0.
\end{equation}\Mpe{Guedjcharact}

Our approach uses a mixture of analytic tools in order to control the asymptotic behavior of Lelong numbers and verify Guedj's condition \eqref{Guedjcharact}.\\

In particular, we use a technique due to J. P. Demailly of approximation of currents by currents with analytic singularities. As a consequence, we are able to reduce the general problem for positive closed (1,1)-currents to the case of the currents of integration of a suitable hypersurface. Applying Dinh-Sibony's result to this hypersurface allows us to obtain Theorem C.\\

For the proof of Theorem A, we proceed as in \cite{MR2032940} and study the Lelong numbers of $d^{-n}[\CC_{f^n}]$, where $\CC_{f^n}$ denotes the critical set of $f^n$. We prove that the (totally invariant) set

$$E:=\bigcup_{\delta>0}\bigcap_{n\in\N}\left\{x\in\P^k\mid\ord_x(\CC_{f^n})\geq\delta d^n\right\}$$ is algebraic, using a refined version, used in \cite{MR2468484}, of certain self-intersection inequalities \'{a} la Demailly.\\

We then proceed inductively on the irreducible components $X\subset E$. The main problem that arises, is that $X$ might be too singular, hence, making sense of the critical set of $f|_X$ is hard. Our assumptions on the singular locus of $X$ will let us get around this problem.\\


{\bf Acknowledments:} I want to thank my advisor Mattias Jonsson for all the help and restless support through the preparation of this article, explaining and correcting many of the ideas exposed here.\\

I would also like to take the opportunity to thank professors Sebastian Bouckson, Charles Favre, Tien-Cuong Dinh and Nessim Sibony for their invaluable explanations and comments. Interacting with all of them -- in Paris and Ann Arbor -- has been of extreme importance for my understanding of the subject.


\section{Background}

Throughout this paper $f:\P^k\to\P^k$ will denote a holomorphic map of degree $d\geq2$, i.e. the map $f$ is given by a $(k+1)$-tuple of homogeneous polynomials of degree $d\geq2$ in $k+1$ variables, having $(0,\ldots,0)\in\C^{k+1}$ as the only common root. For a detailed discussion on the properties of holomorphic (and meromorphic) dynamics on projective spaces we refer the reader to \cite{SibSurvey}.\\

Denoting by $\omega$ the Fubini-Study metric on $\P^k$ and using that $H^2_{\text{DR}}(\P^k;\R)$  is generated by $\{\omega\}$, we see that there exists a smooth function $u:\P^k\to\R$ such that

\begin{equation*}\label{cohom1}
d^{-1}f^*\omega=\omega+dd^cu,
\end{equation*} where $dd^c=\frac{\sqrt{-1}}{2\pi}\partial\bar\partial$. In particular, for every $n\in\N$ we observe that

\begin{equation*}\label{eq1}
d^{-n}(f^n)^*\omega=\omega+dd^cu_n,
\end{equation*} where $u_n:=\sum_{i=0}^{n-1}d^{-i}u\circ f^i$ is a sequence of smooth functions on $\P^k$. Since $u$ is bounded, it follows that the series defining $u_n$ converges uniformly to a - continuous - function $g_f$ on $\P^k$. Hence, the sequence of smooth forms $\omega+dd^cu_n$ converges in the sense of currents to the positive closed (1,1)-current

\begin{equation*}\label{greencurrent1}
T_f:=\omega+dd^cg_f.
\end{equation*}

The current $T_f$ is called the \emph{Green current} associated to $f$ and it plays a central role for the understanding of the dynamics given by the map $f$. It is invariant (i.e. $f^*T_f=dT_f$), it has Lelong numbers zero everywhere and its support equals the Julia set of $f$ (\cite{MR1285389}, \cite{HubbardPapadopol}).\\

As we stated at the introduction, given a positive closed (1,1)-current $S$ we can verify equidistribution (i.e. $d^{-n}(f^n)^*S$ converges to $T_f$ in the sense of currents) if the sequence $\sup_{x\in\P^k}\nu\left(d^{-n}(f^n)^*S,x\right)$ converges to 0 \cite{MR2017143}. We will show that there is a link between equidistribution and the appearence of certain exceptional sets.\\

It will be convenient to replace iterates $f^s$ by $f$. For this we need

\begin{lem}\label{wlog}
Let $S$ be a positive closed (1,1)-current of mass 1. Then the following are equivalent

\begin{itemize}
\item[(i)] $d^{-n}(f^n)^*S\to T_f$;
\item[(ii)] $d^{-ns}(f^{ns})^*S\to T_f$ for some $s\geq1$.
\end{itemize}
\end{lem}\Mp{wlog}

\begin{proof}
Note that (i) $\Rightarrow$ (ii) follows immediately and that (ii) implies  $$d^{-n}(f^n)^*S\to d^{-l}\left(f^l\right)^*T_f$$ for some $l\in\{0,1,\ldots,s-1\}$. The $f$-invariance of $T_f$ gives $d^{-l}\left(f^l\right)^*T_f=T_f$.
\end{proof}


\subsection{Totally invariant algebraic sets}

A subset $X\subset\P^k$ is said to be \emph{totally invariant} if $f^{-s}(X)\subset X$ for some $s\geq1$. If $X\subset\P^k$ is an irreducible algebraic totally invariant set, then it follows that $f^{-s}(X)=X$. Moreover, if $X$ is a totally invariant algebraic set of codimension $p$ in $\P^k$ then the holomorphic map

$$g:=f^s|_X:X\to X$$ has topological degree $d^{sp}$ and $(f^s)^*[X]=d^{sp}[X]$, where $[X]$ denotes the current of integration of $X$.\\

A crucial property of totally invariant algebraic subsets of $\P^k$ is the following (non-trivial) well known fact

\begin{thm}\label{totallyinvariantsetsalgebraic}
The collection of all proper totally invariant algebraic subsets of $\P^k$ is finite.
\end{thm}\Mp{totallyinvariantsetsalgebraic}

This says that the collection of all proper totally invaraint algebraic subsets of $\P^k$ is finite. A proofs of Theorem \ref{totallyinvariantsetsalgebraic} can be found in \cite{MR2468484} (see also \cite{MR1285389} for the case $k=2$).\\

For the more general situation $g:X\to X$ where $g$ is a regular map and $X$ is a projective variety, the same conclusion can be derived from the work of Dinh-Sibony in \cite{MR2468484}, giving us the more useful result: Let $g:X\to X$ be a regular self-map of a projective variety $X$. Then, the collection of all proper totally invariant subsets of $X$ is finite.


\section{Orders of vanishing}

In this section we discuss orders of vanishing on algebraic varieties, where a major difficulty will be to deal with the singular locus. The standard references \cite{MR1492525} and \cite{KollarMori} (see also \cite{ClemensKollarMori88}) contain a details discussion of all the concepts discussed in this section.\\

Let $X$ be an irreducible projective normal variety of dimension $k$, and assume $X$ to be $\Q$-factorial, i.e.\ every Weil divisor is $\Q$-Cartier. If $\pi:Y\to X$ is a birational morphism, we say that $E\subset Y$ is a \emph{divisor over $X$} if $E$ is a smooth prime divisor in $Y$. We say that $E$ \emph{lies over a point $x\in X$} if $\pi(E)=\{x\}$.\\

If $\phi$ is a rational function on $X$ and $E$ a divisor over $X$, we write $\ord_E(\phi)$ for the order of vanishing of $\phi\circ\pi$ along the divisor $E$. Similarly, if $D$ is a Weil divisor on $X$, we set $\ord_E(D):=\frac1{m}\ord_E\left(\pi^*(mD)\right)$ where $m\in\N$ is chosen so that $mD$ is Cartier. As usual, we denote by $K_X$ the canonical divisor class of $X$.\\

Above, we may assume that $\pi:Y\to X$ is a \emph{log-resolution} of $X$, i.e.\ the exceptional locus $\Exc(\pi)$ of $\pi$ has simple normal crossing. In this case, there exists a unique divisor $K_{Y/X}$ on $Y$ supported on $\Exc(\pi)$, the \emph{relative canonical divisor}, which is in the divisor class of $K_Y-\pi^*K_X$. If both $X$ and $Y$ are smooth, then $K_{Y/X}$ is nothing but the effective divisor defined by the Jacobian $\Jac(\pi)$ of $\pi$.\\

Given a prime divisor $E$ over $X$, we define the \emph{log-discrepancy} $a_E$ of $E$ by

\begin{equation}\label{deflogdiscrepancy}
a_E:=\ord_E\left(K_{Y/X}\right)+1.
\end{equation}\Mpe{deflogdiscrepancy}

We say that $X$ is klt (short for Kawamata log-terminal) if $a_E>0$ for every prime divisor $E$ over $X$.\\

Let us give a simple example: Let $\pi:Y\to X$ be the blow-up of $X$ at a smooth point $0\in X$. We can choose local coordinates $(x_1,\ldots,x_k)$ so that our map $\pi$ can be written as

$$\pi(x_1,\ldots,x_k)=(x_1x_2,\ldots,x_{k-1}x_k,x_k),$$ giving us that $\Jac(\pi)=x_k^{k-1}$. Hence, if $E=\pi^{-1}(0)$ is the exceptional prime divisor above $0$, we have that

\begin{equation}\label{logdisc1}
a_E=\ord_E\left(\Jac(\pi)\right)+1=k.
\end{equation}\Mpe{logdisc1}

Recall that if $x\in X$ and $\phi\in\OO_{X,x}$, the order of vanishing $\ord_x(\phi)$ of $\phi$ at $x$ is defined to be 

$$\ord_x(\phi):=\max\left\{s\in\N\mid \phi\in\mmm_x^s\right\}$$ where $\mmm_x$ denotes the maximal ideal of the local ring $\OO_{X,x}$. If $x$ is smooth and $E$ is an exceptional divisor of a single blowup (at $x$), then $\ord_x(\phi)$ and $\ord_E(\phi)$ coincide. However, when $x\in X$ is singular, $\ord_x$ may not be a valuation.\\

The following well known lemma will play an important role in this paper

\begin{lem}\label{tougeronlemma}
Let $E$ be a prime divisor over $X$ above a smooth point $x\in X$. Then

$$\ord_E(D)\leq a_E\ord_x(D)$$ for every divisor $D\subset X$.
\end{lem}\Mp{tougeronlemma}

The proof of Lemma \ref{tougeronlemma} can be found in \cite{MR0440598} p.178 Lemma 1.3.\\

Let $g:X\to X$ be a surjective regular map. Then there exists a unique Weil divisor $\CC_g$ on $X$, the \emph{critical divisor}, whose restriction to $\Xreg\cap g^{-1}(\Xreg)$ equals the Cartier divisor

$$\left\{x\in\Xreg\cap g^{-1}(\Xreg)\mid\Jac(g)=0\right\}.$$ Since $X$ is $\Q$-factorial, $\CC_g$ is $\Q$-Cartier. It belongs to the divisor class of $g^*K_X-K_X$.\\

We need

\begin{lem}\label{valuationstuff1}
Let $\pi:Y\to X$ be a log-resolution and let $E\subset Y$ be a prime divisor. Then, there exists a log-resolution $\pi':Y'\to X$ and a prime divisor $E'\subset Y'$ such that the meromorphic lifting $\bar g:Y\dashrightarrow Y'$ satisfies $\bar g(E)=E'$.
\end{lem}\Mp{valuationstuff1}

\begin{proof}
To any rank 1 valuation $\nu:\C(X)\setminus\{0\}\to\R$ of the function field $\C(X)$ we can associate two basic invariants: the \emph{value group} $$\Gamma_{\nu}:=\left\{\nu(\phi)\mid\phi\in\C(X)\setminus\{0\}\right\}\subset\R$$ and the \emph{residue field} $$\K(\nu):=\{\nu\geq0\}/\{\nu>0\}.$$

A valuation $\nu$ is of the form $r\ord_E$ where $r>0$ and $E$ is a divisor over $X$ if and only if $\Gamma_{\nu}=r\Z$ and $\K(\nu)$ has trascendence degree $\dim(X)-1$ over $\C$ (see \cite[Proposition 10.1]{vaquievaluations}).\\

In our situation, set $\nu:=\ord_E$ and $\nu':=g_*\nu$. Then

$$\Gamma_{\nu'}\subset\Gamma_{\nu}=\Z$$ hence $\Gamma_{\nu'}=r\Z$ for some $r\in\N$.\\

Furthermore, since $g$ is a finite map, $\K(\nu)$ is a finite extension of $\K(\nu')$, therefore both $\K(\nu)$ and $\K(\nu')$ have the same trascendence degree over $\C$, i.e.\ $\dim(X)-1$. From the above description we have that $\nu'=r\ord_{E'}$ for some prime divisor $E'$ over $X$.
\end{proof}

As a consequence of Lemma \ref{valuationstuff1}, we know we can choose $x''\in E$ generic, such that $\bar g$ is holomorphic at $x''$, the critical set $\CC_{\bar g}$ of $\bar g$ is smooth at $x''$ and $\bar g(\CC_{\bar g})$ smooth at $y'=\bar g(x'')$. Picking local coordinates $(z,z_k)$ and $(w,w_k)$ around $x''$ and $y'$ in such way that $E=\{z_k=0\}$ and $E'=\{w_k=0\}$, we see that

$$(w,w_k)=\bar g(z,z_k)=(z,z_k^r).$$ In particular, we notice that

\begin{equation}\label{mvalu}
\ord_E(\CC_{\bar g})=r-1
\end{equation}\Mpe{mvalu} 

and

\begin{equation}\label{mvaluation}
g_*\ord_E=r\ord_{E'}.
\end{equation}\Mpe{mvaluation}

\begin{pro}\label{propequaldiscrep}
With the same notation as above, the following identity holds
$$ra_{E'}=a_E+\ord_E(\CC_g).$$
\end{pro}\Mp{propequaldiscrep}

\begin{proof}
Let us first assume $X$ to be smooth and fix $\omega$ a meromorphic k-form on $X$, i.e.\ $\omega$ can be writen as 

$$\omega(x)=h(x)dx_1\wedge\cdots\wedge dx_k,$$ where $x=(x_1,\ldots,x_k)$ is some local chart and $h$ a meromorphic function on $X$.\\

From the commutative diagram

\begin{equation}\label{comm1}
\xymatrix{\ar @{} [dr] |{\circlearrowright}
E\subset Y \ar[d]_{\pi} \ar@{-->}[r]^{\bar g} & Y'\supset E' \ar[d]^{\pi'}\\
X \ar[r]_{g} & X }
\end{equation}\Mpe{comm1}  we obtain (in local coordinates) that 

\begin{equation}\label{eqdiscleft}
\begin{array}{lcl}
\pi^*g^*\omega&=&\pi^*g^*(h(x)dx_1\wedge\cdots\wedge dx_k)\\
&=&\pi^*(h\circ g(x)\Jac(g)(x)dx_1\wedge\cdots\wedge dx_k)\\
&=&h\circ g\circ \pi(y)\Jac(g)\circ\pi(y)\Jac(\pi)dy_1\wedge\cdots\wedge dy_k.
\end{array}
\end{equation}\Mpe{eqdiscleft}

On the other hand

\begin{equation}\label{eqdiscright}
\begin{array}{lcl}
\bar g^*\pi'^*\omega&=&\bar g^*\pi'^*(h(x)dx_1\wedge\cdots\wedge dx_k)\\
&=&\bar g^*(h\circ \pi'(y)\Jac(\pi')dy_1\wedge\cdots\wedge dy_k)\\
&=&h\circ\pi'\circ\bar g(y)\Jac(\pi')\circ \bar g(y)\Jac(\bar g)(y)dy_1\wedge\cdots\wedge dy_k.
\end{array}
\end{equation}\Mpe{eqdiscright}

Using the identity $\pi^*g^*=\bar g^*\pi'^*$ from \eqref{comm1} in \eqref{eqdiscleft} and \eqref{eqdiscright}, we obtain the following identity of (Cartier) divisors

\begin{equation}\label{divisorsidentity}
\pi^*\CC_g+K_{Y/X}=\bar g^*K_{Y'/X}+\CC_{\bar g}.
\end{equation}

This in particular implies that

\begin{equation*}
\ord_E(\CC_g)+\underbrace{\ord_E(K_{Y/X})}_{a_E-1}=\underbrace{\bar g_*\underbrace{\ord_E(K_{Y'/X})}_{a_{E'}-1}}_{r(a_{E'}-1)}+\underbrace{\ord_E(\CC_{\bar g})}_{r-1}
\end{equation*} by \eqref{mvalu} and \eqref{mvaluation}, giving us that $ra_{E'}=a_E+\ord_E(\CC_g)$.\\

If $X$ is not smooth, we pick a meromorphic k-form $\omega$ on $\Xreg$. Note that $\divi(\omega)$ extends uniquely as a Weil divisor on $X$ and our assumption on $X$, namely, $X$ is $\Q$-factorial, allows us to obtain the identity \eqref{divisorsidentity} in the singular case. The computation then follows identically as in the smooth case.

\end{proof}


\subsection{Comparison of orders of vanishing}

By definition, $x\in X$ is an \emph{isolated quotient singularity} if there exists a finite group $G_x\subset GL(k,\C)$ acting freely on $\C^k\setminus\{0\}$ such that  $$(X,x)\cong(\C^k,0)/G_x.$$

From now on, we will assume that our variety $X$ has at worst isolated quotient singularities. This in particular implies that $X$ is $\Q$-factorial and with klt singularities (see \cite{KollarMori}, Prop. 5.15 and Prop. 5.20).\\

\begin{lem}\label{lemmaklt}
For every $y\in X$ there exists a constant $C_y\geq1$ such that

$$\ord_y\phi\leq\ord_0\phi\circ\varrho\leq C_y\ord_y\phi$$ for every holomorphic germ $\phi\in\OO_{X,y}$. If $y$ is smooth, we can pick $C_y=1$.
\end{lem}\Mp{lemmaklt}

\begin{proof}
Let $\phi\in\OO_{X,y}$ and denote by $t:=\ord_y\phi$. We then have that

$$\phi\in\mmm_y^t\quad\text{and}\quad\phi\notin\mmm_y^{t+1},$$ hence 

$$\phi\circ\varrho\in\varrho^*\mmm_y^t\quad\text{and}\quad\phi\circ\varrho\notin\varrho^*\mmm_y^{t+1}.$$

We can pick $C=C(y)\in\N$ such that for every $l\in\N$ it follows that

$$\mmm_0^{lC}\subset\varrho^*\mmm_y^l\subset\mmm_0^l,$$ hence

$$\phi\circ\varrho\in\mmm_0^t\quad\text{and}\quad\phi\circ\varrho\notin \mmm_0^{(t+1)C}$$

$$\ord_y(\phi)\leq\ord_0(\phi\circ\varrho)\leq(\ord_y(\phi)+1)C\leq2C\ord_y(\phi).$$ 

Take $C_y=2C$ and the inequalities follow.

\end{proof}

If $g:X\to X$ is a holomorphic map with $y=g(x)$ and

$$\varrho:(\C^k,0)\to(X,y)\quad\text{and}\quad\varrho':(\C^k,0)\to(X,x)$$ are the quotient maps of $y$ and $x$ respectively, the map $g\circ\varrho':(\C^k,0)\to(X,y)$ can be lifted to a continuous (hence holomorphic) map $\hat g:(\C^k,0)\to(\C^k,0)$ such that the following diagram commutes

\begin{equation}\label{liftholomorphic}
\xymatrix{\ar @{} [dr]
(\C^k,0)\ar[d]_{\varrho'} \ar[r]^{\hat g}&(\C^k,0)\ar[d]^{\varrho}\\
(X,x)\ar[r]_{g}&(X,y)}
\end{equation}\Mpe{liftholomorphic} In particular, we have that $\CC_{\hat g}=\varrho^*\CC_g$.\\

We state the main result of this section

\begin{thm}\label{cocycleineq1}
There exists a positive constant $C_X\geq1$ independent of $g$ such that the following inequality holds

$$\ord_x\phi\circ g\leq C_X\left(k+\ord_x\CC_{g}\right)\ord_{g(x)}\phi$$ for every $\phi\in\OO_{X,x}$ and every $x\in X$.
\end{thm}\Mp{cocycleineq1}

\begin{proof}
\textbf{Step 1:} Assume $x$, $y=g(x)\in X$ both smooth.\\

Let $\pi:Y\to X$ be the blow-up  of $X$ at $x$ and $E=\pi^{-1}(x)$ the exceptional divisor. By Lemma \ref{valuationstuff1} we can find a birational morphism $\pi':Y'\to X$ and a divisor $E'\subset Y'$ such that the lift $\bar g:Y\dashrightarrow Y'$ of $g$ satisfies $\bar g(E)=E'$ and $g_*\ord_E=r\ord_{E'}$ for some $r\in\N$. Since $\pi(E)=x$ we must have $\pi'(E')=y$. Hence, for every $\phi\in\OO_{X,y}$ we have

$$\ord_x(\phi\circ g)=\ord_E(\phi\circ g)=r\ord_{E'}(\phi)\leq ra_{E'}\ord_y(\phi)$$ where the last inequality follows from Proposition \ref{tougeronlemma}.\\

By Proposition \ref{propequaldiscrep} we have that 

$$ra_{E'}=a_E+\ord_E(\CC_g),$$ where $a_E=k$ by \eqref{logdisc1}. Therefore
 
\begin{equation}\label{smooth1} 
\ord_x(\phi\circ g)\leq \left(k+\ord_x(\CC_g)\right)\ord_y(\phi).
\end{equation}\Mpe{smooth1}

\textbf{Step 2:} The general case follows using diagram \eqref{liftholomorphic}. Let $\phi\in\OO_{X,y}$. For the left inequality of Theorem \ref{cocycleineq1} we have that

$$\ord_x(\phi\circ g)\leq\ord_0(\phi\circ g\circ\varrho')=\ord_0(\phi\circ\varrho\circ\hat g).$$

Inequality \eqref{smooth1} implies that

$$\ord_0(\phi\circ\varrho\circ\hat g)\leq\left(k+\ord_0(\Jac(\hat g))\right)\ord_0(\phi\circ\varrho)$$ and by Lemma \ref{lemmaklt} we have that

$$\ord_0(\phi\circ\varrho)\leq C_x\ord_{g(x)}(\phi)$$ for some $C_x\geq1$. Note that $$\ord_0(\Jac(\hat g))=\ord_0\CC_{\hat g}=\ord_0(\varrho'^*\CC_g)\leq C_x'\ord_x\CC_g$$ for some $C_x'\geq1$. Taking $C_X:=\max_{x\in X}\{C_xC'_x\}$ we obtain the desired inequality.\\

\end{proof}


\section{The Jacobian cocycle}

In this section we proceed to define one of our key tools, the Jacobian cocycle. For an extensive discussion on (analytic) cocycles we refer the reader to \cite{FavreThesis}, \cite{MR1724404} and \cite{MR2513537}.

\subsection{Definition and properties of the Jacobian cocycle}

Let $X$ be an irreducible normal projective variety with at worst isolated quotient singularities and let $g:X\to X$ be a surjective holomorphic self-map. Then $X$ is $\Q$-factorial and klt (see Section 3). For every $n\in \N$ we denote by $\mu_n^X$ the Zariski usc function 

$$X\ni x\mapsto\mu_n^X(x):=C_X(\dim(X)+\ord_x\CC_{g^n})$$ on $X$, with $C_X\geq1$ as in Theorem \ref{cocycleineq1}. If $Z\subset X$ is an irreducible algebraic subset, we denote by $\mu_n^X(Z)$ the generic value of $\mu_n^X$ on $Z$ given by $$\mu_n^X(Z):=\min_{x\in Z}\left\{\mu_n^X(x)\right\}=C_X\left(\dim(X)+\ord_Z(\CC_{g^n})\right).$$

It is easy to see that the identity

\begin{equation}\label{criticalcocycle2}
\CC_{g^{n+m}} = \CC_{g^n}+(g^n)^*\CC_{g^m}
\end{equation}\Mpe{criticalcocycle2} follows on a suitable Zariski open subset of $X$ for every $n$, $m\in\N$. Since $X$ is $\Q$-factorial, the identity extends to all of $X$ as $\Q$-Cartier divisors.\\

\begin{pro}\label{propwelldefcocycle}
The following is true for the family $\left(\mu_n^X\right)_{n\in\N}$\\
\begin{itemize}\Mp{propwelldefcocycle}

\item[(i)] (Comparison) Let $D$ be a divisor in $X$. Then for every $x\in X$,

$$\ord_x\left((g^n)^*D\right)\leq\mu_n^X(x)\ord_{g^n(x)}(D),$$

\item[(ii)] (Submultiplicativity) for every $n$, $m\in\N$ and for every $x\in X$ the following inequality holds
$$\mu^X_{n+m}(x)\leq\mu^X_n(x)\mu^X_m\left(g^n(x)\right).$$

\end{itemize}

\end{pro}

\begin{proof}

Part (i) follows immediately from Theorem \ref{cocycleineq1}.\\

For proving (ii), observe that from \eqref{criticalcocycle2} we obtain

\begin{multline*}
\mu_{n+m}^X(x)=C_X(\dim(X)+\ord_x(\CC_{g^{n+m}}))=\\=C_X(\dim(X)+\ord_x(\CC_{g^n})+\ord_x((g^n)^*\CC_{g^m})).
\end{multline*}

By Theorem \ref{cocycleineq1} we have that

$$\ord_x((g^n)^*\CC_{g^m})\leq C_X\left(\dim(X)+\ord_x(\CC_{g^n})\right)\ord_{g^n(x)}(\CC_{g^m})$$ implying that

\begin{multline*}
\mu_{n+m}^X(x)\leq \left(C_X(\dim(X)+\ord_x(\CC_{g^n}))\right)\left(\frac1{C_X}+\ord_{g^n(x)}(\CC_{g^m})\right)\leq\\\leq\mu_n^X(x)\mu_m^X(g^n(x)).
\end{multline*}

\end{proof}

By the submultiplicativity property, it is easy to see that the function

$$X\ni x\mapsto \mu_{\infty}^X(x):=\lim_{n\to+\infty}\left(\mu_n^X(x)\right)^{\frac1{n}}$$ is well defined (i.e. the limit always exist for every $x\in X$). It satisfies $$\mu_{\infty}^X\circ g=\mu_{\infty}^X.$$

\subsection{Jacobian cocycles and totally invariant sets}

Let $f:\P^k\to\P^k$ be a holomorphic map of algebraic degree $d\geq2$ and let $X\subset\P^k$ be an irreducible algebraic set such that $f^{-1}(X)=X$. Define $g:=f|_X:X\to X$ and take $\tilde g:\tilde X\to \tilde X$ to be the lift of $g$ to the normalization $\pi:\tilde X\to X$. The commutative diagram

\begin{equation*}
\xymatrix{\ar @{} [dr] 
\tilde X\ar[r]^{\tilde g}\ar[d]_{\pi}&\tilde X\ar[d]^{\pi}\\
X\ar[r]_{g}\ar[d]_{\iota}&X\ar[d]^{\iota}\\
\P^k\ar[r]_{f}&\P^k
}\end{equation*} where $\iota:X\to\P^k$ denotes the inclusion map, gives us the following commutative diagram of groups and homomorphisms

\begin{equation*}
\xymatrix{\ar @{} [dr] 
\Z\cong H^2(\P^k;\Z)\ar[r]^{f^* =\, d\cdot}\ar[d]_{(\iota\pi)^*}& H^2(\P^k;\Z)\cong\Z\ar[d]^{(\iota\pi)^*}\\
H^2(\tilde X;\Z) \ar[r]_{\tilde g^*} & H^2(\tilde X;\Z) }
\end{equation*} where $\iota:X\hookrightarrow\P^k$ In particular, if $\omega$ is the Fubini-Study metric on $\P^k$ then $\{\omega\}$ generates $H^2(\P^k,\Z)$ and $f^*\{\omega\}=d\{\omega\}$, therefore

$$\tilde g^*\pi^*\iota^*\{\omega\}=\pi^*g^*\iota^*\{\omega\}=\pi^*\iota^*f^*\{\omega\}=d\cdot\pi^*\iota^*\{\omega\}$$ giving us a $\tilde g^*$-invariant class $\{\omega_{\tilde X}\}:=\{(\iota\pi)^*\omega\}$ in $H^2(\tilde X;\Z)$.The class $\{\omega\}$ represents the first Chern class of the ample line bundle $\OO(1)$ in $\P^k$, which induces an ample line bundle $\OO_{\tilde X}(1):=(\iota\pi)^*\OO(1)$ on $\tilde X$ with

$$\tilde g^*c_1\left(\OO_{\tilde X}(1)\right)=\tilde g^*\{\omega_{\tilde X}\}=d\{\omega_{\tilde X}\}.$$

The ample classes on $\tilde X$ form a strict open convex cone which is invariant by $\tilde g^*$ in the finite dimensional vector space $H^2(\tilde X;\R)$ (for details see \cite{GH}, \cite{MR2095471}). Using that there exists an invariant ample class on $\tilde X$, namely $\{\omega_{\tilde X}\}$ with $(\tilde g^*)^n\{\omega_{\tilde X}\}=d^n\{\omega_{\tilde X}\}$ for all $n\in\N$, it is possible to conclude that $\|(\tilde g^*)^n\|\lesssim d^n$ for every $n\in\N$.\\

Moreover, if we assume $\tilde X$ to have at worst isolated quotient singularities then we obtain

\begin{lem}\label{boundedmassforthecriticalset}
In the same setting as above, there exists a positive constant $A$ independent of $x\in\tilde X$ and $n\in\N$ such that $$\mu_n^{\tilde X}(x)\leq A d^n.$$ 
\end{lem}\Mp{boundedmassforthecriticalset}

\begin{proof}

Note that $$\CC_{\tilde g^n}=\sum_{i=0}^{n-1}(\tilde g^i)^*\CC_{\tilde g}$$ (see equation \eqref{criticalcocycle2}).\\ 

Denote by $A':=\sup_{x\in\tilde X}\ord_x(\CC_{\tilde g})$. Then we have that

\begin{multline*}
\ord_x(\CC_{\tilde g^n})\leq A'\sum_{i=0}^{n-1}d^i=A'\frac{d^n-1}{d-1}\Longrightarrow\\\Longrightarrow\mu^X_n(x)\leq C_X\left(\dim(X)+A'\frac{d^n-1}{d-1}\right)\leq\\\leq C_X\left(\frac{\dim(X)}{d^n}+A'\frac{1-d^{-n}}{d-1}\right)d^n\leq (\underbrace{\dim(X)+A'(d-1)^{-1}}_{A})d^n.
\end{multline*}

\end{proof}

\begin{thm}\label{algebraictotinv}
The $\tilde g$-totally invariant set

$$E_{\tilde X}:=\left\{x\in \tilde X\mid \mu_{\infty}^{\tilde X}(x)=d\right\}=\bigcup_{\delta>0}\bigcap_{n\in\N}\left\{x\in \tilde X\mid \ord_x\left(\CC_{{\tilde g}^n}\right)\geq\delta d^n\right\}$$ is algebraic.
\end{thm}\Mp{algebraictotinv}

The proof relies on the fact that the family of totally invariant algebraic subsets of $\tilde X$ is finite (see Section 2). The key idea is to prove that every irreducible component of $E_{\tilde X}$ is totally invariant for some iterate of $g$, therefore $E_{\tilde X}$ has only finitely many components. In order to do this we use a uniform bound from \cite{Rodrigo1} for the orders of vanishing of $d^{-N}[\CC_{{\tilde g}^N}]$.

\begin{proof}

We argue by contradiction:  Let $Z\subset E_{\tilde X}$ be an irreducible component; define $S_N:=d^{-N}[\CC_{\tilde g^N}]$ for $N$ large ($N$ will be made explicit later). By definition there exists $\delta>0$ independent of $N$ such that

$$\ord_x(S_N)\geq\delta,\quad\forall\,x\in Z.$$

For every $0\leq r<N$, by the chain rule $\CC_{\tilde g^N}=\CC_{\tilde g^{r+N-r}}=\CC_{\tilde g^r}+(\tilde g^r)^*\CC_{\tilde g^{N-r}}$ (see equation \eqref{criticalcocycle2}) we obtain that for every $x\in\tilde g^{-r}(Z)$	

$$\ord_x\CC_{\tilde g^N}\geq \ord_x\left((\tilde g^r)^*\CC_{\tilde g^{N-r}}\right)\geq\ord_{\tilde g^r(x)}\CC_{\tilde g^{N-r}}\geq\delta d^{N-r}$$ implying

$$\ord_x(S_N)\geq\delta d^{-r},\quad\forall\,x\in {\tilde g}^{-r}(Z).$$

Let $Y$ be the minimal irreducible algebraic set containing $Z$ which is totally invariant by $\tilde g^s$ for some $s\geq1$. For simplicity we assume $s=1$. If $Z\neq Y$, then $Z$ has positive codimension $p>0$ in $Y$ and since $Y\not\subset E_{\tilde X}$ we can find $C>0$ and $\lambda<d$ such that $\ord_{Y}(\CC_{\tilde g^n})\leq C\lambda^n$ for all $n\geq1$. Then it follows that

$$\ord_x(S_N)\leq C\left(\frac{\lambda}{d}\right)^N\ll1$$ for every $x\in Y$. Denote by $\beta$ the generic Lelong number of $S_N$ along $Y$. Thus $0\leq\beta\leq C(\lambda/d)^N$.\\

Since $Z$ is not totally invariant, following the ideas of Dinh-Sibony (see \cite{MR2468484}, Lemma 6.10) we know there exists a constant $\theta>0$ (independent of $r$) and algebraic sets $Z_r\subset f^{-r}(Z)$ of degrees $d_r$ satisfying

$$d_r\geq\theta d^{rp}\quad\forall\,r\geq1.$$

Define $Z'_0:=Z_0$ and $Z'_r:=\overline{Z_r\setminus(Z_0\cup\cdots\cup Z_{r-1})}$ for $r>0$; note that for $r\neq s$, $Z'_r$ and $Z'_s$ have no common irreducible components. Denote by $d'_r$ the degree of $Z'_r$. It is clear by the construction that $d'_0+\ldots+d'_r\geq d_r$ and that the generic Lelong numbers $\nu_r$ of $S_N$ at $Z'_r$ satisfy $\nu_r\geq\delta d^{-r}$. We have:

\begin{lem}\label{OurInequality}
There exist a positive constant $A_{\tilde X}$ independent of $r\geq1$ and $N$ such that

$$\sum_{r=0}^M(\nu_r-\beta)^pd'_r\leq A_{\tilde X}$$
\end{lem}\Mp{OurInequality}

Lemma \ref{OurInequality} can be viewed as a generalization of a special case of Demailly's self-intersection inequalities found in \cite{DemRegCurr} and \cite{DemNumCrit}. For a proof of Lemma \ref{OurInequality} (in a more general setting) see \cite{Rodrigo1}.\\

We now fix $M<N$ (which can be made very large) such that

$$\beta d^r\leq\frac1{2}\delta,\quad\forall\,r=0,\ldots,M.$$

We observe that

\begin{multline}\label{keyinequality1}
\sum_{r=0}^M(\nu_r-\beta)^pd'_r\geq\sum_{r=0}^M(\delta d^{-r}-\beta)^pd'_r\geq\\\geq\left(\frac{\delta}{2}\right)^p\sum_{r=0}^Md^{-rp}d'_r=\\=\left(\frac{\delta}{2}\right)^p[d'_0(1-d^{-p})+(d'_0+d'_1)(d^{-p}-d^{-2p})+\ldots\\ \ldots +(d'_0+\ldots+d'_{M-1})(d^{-(M-1)p}-d^{-Mp})+d_Md^{-Mp}].
\end{multline}\Mpe{keyinequality1}
Notice that for every $r=0,\ldots,M-1$ we have

$$d^{-rp}-d^{-(r+1)p}\geq\frac1{2}d^{-rp}.$$

Plugging this into \eqref{keyinequality1} we obtain that

\begin{multline}\label{keyinequality2}
\sum_{r=0}^M(\nu_r-\beta)^pd'_r\geq\frac1{2}\left(\frac{\delta}{2}\right)^p\sum_{r=0}^Md^{-rp}d_r\geq\\\geq \frac1{2}\left(\frac{\delta}{2}\right)^p\sum_{r=0}^Md^{-rp}\theta d^{rp}=\frac{\theta}{2}\left(\frac{\delta}{2}\right)^p(M+1),
\end{multline}\Mpe{keyinequality2}
where $\delta>0$ and $\theta>0$ are independent of $M$ and $N$. Therefore, using Lemma \ref{OurInequality} in the inequality \eqref{keyinequality2} we obtain

$$\frac{\theta}{2}\left(\frac{\delta}{2}\right)^p(M+1)\leq A_{\tilde X}$$ which produces a contradiction if we take $N$ and $M$ sufficiently large.

\end{proof}

\begin{cor}\label{imageexceptional}
The algebraic set $E_X:=\pi\left(E_{\tilde X}\right)\subset X$ is totally invariant: $$g^{-1}(E_X)=E_X.$$
\end{cor}\Mp{imageexceptional}

\begin{proof}[Proof of Corollary \ref{imageexceptional}]
We prove that every component of $E_X$ is totally invariant.\\

Let $Z\subset E_X$ be an irreducible component and write

$$\pi^{-1}(Z)=\tilde Z\cup \tilde Z',$$ where $\tilde Z\subset E_{\tilde X}$ and $\pi(\tilde Z)=Z$. Then there exists $l\geq1$ such that $\tilde g^{-l}(\tilde Z)=\tilde Z$, this in particular implies that

$$g^l(Z)=g^l(\pi(\tilde Z))=\pi(\tilde g^l(\tilde Z))=\pi(\tilde Z)=Z\Longrightarrow Z\subset g^{-l}(Z).$$ Write $g^{-l}(Z)=Z\cup W$.\\

If $W\neq\emptyset$ we have that $g^l(W)=Z$. On the other hand,

$$\tilde g^{-l}\pi^{-1}(Z)=\tilde g^{-l}(\tilde Z\cup\tilde Z')=\tilde Z\cup \tilde g^{-l}(\tilde Z')$$

and $$\tilde g^{-l}\pi^{-1}(Z)=\pi^{-1}g^{-l}(Z)=\pi^{-1}(Z\cup W)=\tilde Z\cup\tilde Z'\cup\pi^{-1}(W).$$ Putting this together we obtain that

$$\tilde g^{-l}(\tilde Z')=\tilde Z'\cup\pi^{-1}(W)\Rightarrow \tilde Z'=\tilde g^l(\tilde g^{-l}(\tilde Z'))=\tilde g^l(\tilde Z')\cup \tilde g^l(\pi^{-1}(W)).$$

Since the map $g:=f|_X:X\to X$ is open, it is easy to see that

\begin{equation}\label{propfollownormal}
{\tilde g}(\pi^{-1}(W))=\pi^{-1}(g(W))
\end{equation} for every $W\subset X$. In particular, by identity \eqref{propfollownormal} we have that

$$\tilde g^l(\pi^{-1}(W))=\pi^{-1}(g^l(W))=\pi^{-1}(Z)=\tilde Z\cup\tilde Z'$$ implying

$$\tilde Z'=\tilde g^l(\tilde Z')\cup\tilde Z\cup\tilde Z'\Longrightarrow\tilde Z\subset\tilde Z'$$ contradicting our hypothesis. Hence $W=\emptyset$ and therefore $g^{-l}(Z)=Z$.

\end{proof}


\subsection{The exceptional family $\EE_f$}

\begin{dfn}\label{defexceptionalset}
We define the \emph{exceptional family $\EE_f$ of $f$} as the finite collection of irreducible subsets $X\subseteq\P^k$ such that\\

\begin{itemize}
\item[(i)] $\P^k\in\EE_f$;\\

\item[(ii)] $X\in\EE_f\setminus\{\P^k\}$ if and only if there exist $X'\in\EE_f$ such that $X$ is an irreducible component of $E_{X'}$. In this case we will say that $X$ is an immediate successor of $X'$.\\
\end{itemize}

The exceptional family $\EE_f$ of $f$ is a partially ordered set, where $X\preceq Y$ if there exist a sequence of elements $X=X_1\subsetneq\cdots\subsetneq X_r=Y$ in $\EE_f$ such that $X_i$ is an immediate successor of $X_{i+1}$ for all $i=1,\ldots,r-1$.
\end{dfn}\Mp{defexceptionalset}

Note that by definition, we have that

$$\emptyset\preceq X\preceq\P^k\quad\forall X\in\EE_f.$$

We will say that $X\in\EE_f$ is an \emph{exceptional leaf} if $\emptyset$ is the immediate succesor of $X$ (i.e. $E_X=\emptyset$).


\subsection{Asymptotic behavior}

We finish this section giving a uniform estimate of the orders of vanishing outside the totally invariant algebraic subset $E_{\tilde X}\subset\tilde X$.

\begin{thm}\label{asymptotic1}
There exist constants $C>0$ and $0\leq\rho<d$ such that

$$\sup_{x\notin E_{\tilde X}}\mu_n^{\tilde X}(x)\leq C\rho^n$$ for all $n\in\N$.
\end{thm}\Mp{asymptotic1}

\begin{cor}\label{asymptotic2}
Given any hypersurface $H$ in $X$, it follows that

$$\sup_{x\notin E_X}d^{-n}\ord_x\left((g^n)^*H\right)\to0$$ as $n\to+\infty$.

\end{cor}\Mp{asymptotic2}

\begin{proof}[Proof Corollary \ref{asymptotic2}]
Note that for every $x\in X$ we have

$$\ord_x\left((g^n)^*H\right)\leq\max_{y\in\pi^{-1}(x)}\ord_y\left(({\tilde g}^n)^*(\pi^*H)\right)$$ implying

$$\sup_{x\notin E_X}\ord_x\left((g^n)^*H\right)\leq\sup_{y\notin\pi^{-1}(\pi(E_{\tilde X})}\ord_y\left(({\tilde g}^n)^*(\pi^*H)\right)\leq\sup_{y\notin E_{\tilde X}}\ord_y\left(({\tilde g}^n)^*(\pi^*H)\right).$$

By Proposition \ref{propwelldefcocycle} (ii), it follows that

$$\sup_{y\notin E_{\tilde X}}\ord_y\left(({\tilde g}^n)^*(\pi^*H)\right)\leq\sup_{y\notin E_{\tilde X}}\mu^{\tilde X}_n(y)\ord_{{\tilde g}^n(y)}H$$ which combined with Theorem \ref{asymptotic1} gives us

$$\sup_{x\notin E_X}d^{-n}\ord_x\left((g^n)^*H\right)\leq C\left(\frac{\rho}{d}\right)^n$$ for some $C>0$ and $\rho<d$. Taking $n\to+\infty$ we obtain the desired convergence to zero.

\end{proof}

\begin{proof}[Proof Theorem \ref{asymptotic1}]
Let ${\tilde X}_{0,1},\ldots,{\tilde X}_{0,m_0}$ be the irreducible components of the critical set $\CC_{\tilde g}$ not contained in $E_{\tilde X}$. For every $i=1,\ldots,m_0$ we can pick $x_{0,i}\in {\tilde X}_{0,i}$ such that there exist $C_1>0$ and $\lambda_1<d$ satisfying

$$\max_{i=1,\ldots,m_0}\{\mu_n^{\tilde X}(x)\}<C_1\lambda_1^n,\quad\forall\,n\geq1.$$

For $N>1$ large, define the algebraic set

$${\tilde X}_1:=\{x\in {\tilde X}\mid\mu_N^{\tilde X}(x)\geq C_1\lambda_1^N\}.$$

We clearly have the proper inclusion of algebraic sets

$${\tilde X}_1\subsetneq\CC_{\tilde g^N},$$ where the codimension of ${\tilde X}_1$ in $\CC_{\tilde g^N}$ is $\geq1$ at every point $x\in {\tilde X}_1\setminus E_{\tilde X}$.\\

If ${\tilde X}_1\subset E_{\tilde X}$, for $n\gg N$ and $x\in {\tilde X}\setminus E_{\tilde X}$, denoting $n=tN+l$, $l\in\{0,\ldots,N-1\}$ we have

$$\mu^{\tilde X}_{n}(x)=\mu^{\tilde X}_{tN+l}(x)\leq\mu^{\tilde X}_{l}(x)\mu^{\tilde X}_{tN}(\tilde g^l(x))\leq\mu^{\tilde X}_{N}(x)\prod_{j=0}^{t-1}\mu^{\tilde X}_{N}(\tilde g^{l+jN}(x)).$$

Since ${\tilde X}\setminus E_{\tilde X}$ is totally invariant

$$x\notin E_{\tilde X}\Longrightarrow \tilde g^{l+jN}(x)\notin E_{\tilde X},$$ hence $\tilde g^{l+jN}(x)\notin {\tilde X}_1$ implying that

$$\mu^{\tilde X}_{n}(x)\leq(\dim({\tilde X})+C_1\lambda_1^N)(\dim({\tilde X})+C_1\lambda_1^N)^t\leq2C_1^{t+1}\lambda_1^{n-l-1}.$$

Now we can find $A>0$ and $\lambda_1<\rho<d$ (independent of $x$ and $n$) such that

$$2C_1^{t+1}\lambda_1^{n-l-1}\leq A\rho^n,\quad\forall\,n\geq1,$$ and the theorem follows.\\

Now assume ${\tilde X}_1\not\subset E_{\tilde X}$. Let ${\tilde X}_{1,1},\ldots,{\tilde X}_{1,m_1}$ be the irreducible components of ${\tilde X}_1$ that are not contained in $E_{\tilde X}$. As before, for every $i=1,\ldots,m_1$ we can pick $x_{1,i}\in X_{1,i}$ and $C_2\geq C_1$, $\lambda_1<\lambda_2<d$ such that

$$\max_{i=1,\ldots,m_1}\{\mu_n^{\tilde X}(x)\}<C_2\lambda_2^n,\quad\forall\,n\geq1.$$

Define the algebraic set

$${\tilde X}_2:=\{x\in {\tilde X}\mid\mu_N^{\tilde X}(x)\geq C_2\lambda_2^N\},$$ which has codimension $\geq2$ in $\CC_{\tilde g}$ at every point $x\in {\tilde X}_2\setminus E_{\tilde X}$. Again, if ${\tilde X}_2\subset E_{\tilde X}$ the theorem follows picking some $A>0$ and $\lambda_2<\rho<d$, so we can assume ${\tilde X}_2\not\subset E_{\tilde X}$. Inductively we construct a strictly decreasing sequence of algebraic sets

$${\tilde X}_j:=\{x\in {\tilde X}\mid\mu_N^{\tilde X}(x)\geq C_j\lambda_j^N\},$$ for $j=1,\ldots,\dim({\tilde X})$, where $\lambda_1<\lambda_2<\cdots<\lambda_{\dim({\tilde X})}<d$, $0<C_1\leq C_2\leq\cdots\leq C_{\dim({\tilde X})}$ and the codimension of ${\tilde X}_j$ in $\CC_{\tilde g}$ is $\geq j$ at every point $x\in {\tilde X}_j\setminus E_{\tilde X}$. Thus, there exists $1\leq j_0\leq \dim({\tilde X})$ such that ${\tilde X}_{j_0}\subset E_{\tilde X}$ (since ${\tilde X}_{\dim({\tilde X})}\setminus E_{\tilde X}=\emptyset$), implying that there exist $A\geq C_{j_0}$ and $\lambda_{j_0}<\rho<d$ so that

$$\mu_n^{\tilde X}(x)\leq2C_{j_0}^{t+1}\lambda_{j_0}^{n-l-1}\leq A\rho^n,\quad\forall\,n\geq1,$$ for every $x\in {\tilde X}_{j_0}\setminus E_{\tilde X}$ as before. This concludes the proof.

\end{proof}


\section{Approximation of (1,1)-currents}

In this section we will discuss a method introduced by J. P. Demailly for approximating positive closed (1,1)-currents by currents with controlled singular sets. This approximation method will play a very important role in our approach.\\

Let $S$ be a positive closed (1,1)-current on $\P^k$ of mass 1, i.e. 

$$S=\omega+dd^c\varphi$$ where $\omega$ is the Fubini-Study metric on $\P^k$ and $\varphi$ is a $\omega$-psh function on $\P^k$. The metric $\omega$ is the curvature form of some smooth metric $h$ on $\OO(1)$ over $\P^k$, hence $S$ is the curvature form of the singular metric $\hat h=he^{-\varphi}$ on $\OO(1)$.\\

For every $m\in\N$, let $\HH_m$ be the Hilbert space defined by

$$\HH_m:=\left\{\sigma\in H^0\left(\P^k,\OO(m+1)\right)\mid\|\sigma\|^2_{\hat h^{\otimes m}\otimes h}<+\infty\right\}$$ where

$$\|\sigma\|^2_{\hat h^{\otimes m}\otimes h}=\int_{\P^k}\hat h^{\otimes m}\otimes h(\sigma)\frac{\omega^k}{k!}=\int_{\P^k}h^{\otimes (m+1)}(\sigma)e^{-2m\varphi}\frac{\omega^k}{k!}.$$ Given any orthonormal basis $\{\sigma_{m,j}\}_{j=1}^{N_m}$ of $\HH_m$ (note that $N_m\leq(k+m)!/m!k!$), define the function

\begin{equation}\label{varphisubm}
\P^k\ni x\mapsto \varphi_m(x):=\frac{1}{2m}\log\left(\sum_{j=1}^{N_m} h^{\otimes(m+1)}(\sigma_{m,j})(x)\right)
\end{equation}\Mpe{varphisubm} and define the sequence of (1,1)-currents with analytic singularities

$$S_m:=\omega+dd^c\varphi_m$$ on $\P^k$.\\

The following result is due to J. P. Demailly

\begin{thm}\label{analyticapproximation}
With the same notation as above, the following holds\\

\begin{itemize}
\item[(i)] $S_m\geq-\frac{1}{m}\omega$ (controlled loss of positivity);\\

\item[(ii)] The sequence $S_m$ converges weakly to $S$;\\

\item[(iii)] For every $x\in X$ the Lelong numbers at $x$ satisfy

$$0\leq\nu(S,x)-\nu(S_m,x)\leq\frac{k}{m},$$ where $\nu(S_m,x)$ is defined as the Lelong number of $S_m+\frac1{m}\omega$ at $x$. In particular, the Lelong numbers $\nu(S_m,x)$ converge uniformly to $\nu(S,x)$.
\end{itemize}
\end{thm}\Mp{analyticapproximation}

Theorem \ref{analyticapproximation} holds in greater generality. See \cite{DemNumCrit} for $X$ a projective manifold and $S$ in the cohomology class of a line bundle, and \cite{DemRegCurr} for $X$ K\"ahler and $S$ any (almost) positive closed (1,1)-current. We study this theorem in \cite{Rodrigo1} and find applications for $X$ a singular variety that we used for the proof of Lemma \ref{OurInequality}.\\

Let us sketch a proof of Theorem \ref{analyticapproximation} as we will need some of the ingredients. Inequality (i) follows immediately from the definition of $\varphi_m$. For the comparison of the Lelong numbers in (iii), the main ingredients are the Ohsawa-Takegoshi extension theorem for the left inequality and the main value inequality for the right one. More precisely, the Ohsawa-Takegoshi theorem yields a constant $C>0$, such that for every $x\in\P^k$ we can find a global section $\sigma$ of $\OO(m+1)$ with the properties 

$$\|\sigma\|_m\leq C\quad\text{and}\quad\hat h^{\otimes m}\otimes h(\sigma)(x)=1$$ provided that $\hat h(x)\neq 0$. Using that $$\varphi_m(x)=\frac1{2m}\log\sup_{\|\sigma\|_m=1} h^{\otimes(m+1)}(\sigma)(x)$$ it is a straighforward computation to check that

$$e^{2m\varphi(x)}=h^{\otimes(m+1)}(\sigma)(x)\|\sigma\|_m^2\leq C^2e^{2m\varphi_m(x)}$$ which implies 

\begin{equation}\label{OTimplication}
\varphi(x)\leq\varphi_m(x)+\frac{C'}{m}.
\end{equation}\Mpe{OTimplication}

For the second inequality, we trivialize $\OO(1)$ around $x\in\P^k$ and using the mean value inequality on the holomorphic sections of $\OO(m+1)$ it is not hard to see that

\begin{equation}\label{OTimplication2}
\varphi_m(x)\leq\sup_{B(x,r)}\varphi+C''(x,r)+\frac{C'''}{m}
\end{equation}\Mpe{OTimplication2} for $r>0$ small, where $\lim_{r\to0}C''(x,r)=0$.\\

Finally, (ii) follows from \eqref{OTimplication} and \eqref{OTimplication2}.\\

We will need a version of Theorem \ref{analyticapproximation}, (iii) which is preserved by the dynamical system $f:\P^k\to\P^k$, where $f$ is holomorphic of degree $d\geq2$. We prove:

\begin{pro}\label{control1}
For every $x\in \P^k$ and for every $n$, $m\in\N$ we have

$$0\leq\nu((f^n)^*S,x)-\nu((f^n)^*S_m,x)\leq (k+1)\frac{d^n}{m}.$$
\end{pro}\Mp{control1}

\begin{proof}
It suffices to prove this for $n=1$. The left inequality uses Ohsawa-Takegoshi extension theorem following the exact same argument as in \eqref{OTimplication}. We prove then the right inequality.\\	

Let $\pi:Y\to\P^k$ be the blowup of $\P^k$ at $x\in\P^k$ and $E=\pi^{-1}(x)$ the exceptional divisor above $x$, hence

$$\nu(S,x)=\nu(S,E).$$ 

By Lemma \ref{valuationstuff1} we can find an exceptional divisor $E'$ over $f(x)\in\P^k$ such that $f_*\ord_E=r\ord_{E'}$ for some $r\in\N$. We therefore obtain

$$\nu(f^*S,x)=\nu(f^*S,E)=r\nu(S,E')$$ implying that

\begin{equation}\label{ineqr} 
0\leq\nu(f^*S,x)-\nu(f^*S_m,x)=r\left(\nu(S,E')-\nu(S_m,E')\right).
\end{equation}\Mpe{ineqr}

We need the following strengthening of Theorem \ref{analyticapproximation} (iii)

\begin{lem}\label{bfj}
Let $\pi:Y\to\P^k$ be a modification of $\P^k$ at $x\in\P^k$ and let $E\subset\pi^{-1}(x)$ be an exceptional divisor. Then, with $S$ and $S_m$ as before, we have

$$0\leq\nu(S,E)-\nu(S_m,E)<\frac{a_E}{m}$$ for every $m\in\N$, where $a_E$ denotes the log-discrepancy of $E$.
\end{lem}\Mp{bfj}

The result given by Lemma \ref{bfj} can be found in \cite{BFJ08}, p.486. We sketch a proof of this 

\begin{proof}
Pick coordinates $(x_1,\ldots,x_k)$ around a general point of $E\subset Y$ and write $E=\{x_1=0\}$ (locally) around this point.\\

If $S=\omega+dd^c\varphi$, then 

\begin{equation}\label{bfjineq1}
\varphi\circ\pi\leq\nu(S,E)\log|x_1|+O(1).
\end{equation}\Mpe{bfjineq1} Given a local section $\sigma$ at $x\in\P^k$ of some element of $\HH_m$, we have that

$$\int_{U}|\sigma|^2e^{-2m\varphi}<+\infty$$ in a neighborhood $U$ of $x$. Therefore

\begin{equation}\label{bfjineq2}
\int_{\pi^{-1}(U)}|\sigma\circ\pi|^2e^{-2m\varphi\circ\pi}|\Jac(\pi)|^2<+\infty.
\end{equation}\Mpe{bfjineq2}

By the definition of $a_E$ (see Section 3) we obtain that

\begin{equation}\label{bfjineq3}
|\Jac(\pi)|^2\sim|x_1|^{2(a_E-1)}
\end{equation}\Mpe{bfjineq3} around $E$, and in a general point of $E$ we have 

\begin{equation}\label{bfjineq4}
|\sigma\circ\pi|^2\sim|x_1|^{2\ord_E(\sigma)}.
\end{equation}\Mpe{bfjineq4}

Putting \eqref{bfjineq1}, \eqref{bfjineq2}, \eqref{bfjineq3} and \eqref{bfjineq4} together we obtain that

$$\int_{\pi^{-1}(U)}|x_1|^{2\ord_E(\sigma)-2m\nu(S,E)+2(a_E-1)}<+\infty$$ which implies (by Fubini's theorem)

\begin{equation}\label{bfjineq5}
\ord_E(\sigma)-m\nu(S,E)+a_E>0
\end{equation}\Mpe{bfjineq5} for all $\sigma\in\HH_m$. Dividing equation \eqref{bfjineq5} by $m$ and taking the maximum over $\{\sigma_1,\ldots,\sigma_{N_m}\}$ an orthonormal basis of $\HH_m$, we finally obtain

$$\nu(S_m,E)-\nu(S,E)+\frac{a_E}{m}>0$$ which concludes the proof.

\end{proof}

Now, using Lemma \ref{bfj} into equation \eqref{ineqr} we obtain that

$$r\left(\nu(S,E')-\nu(S_m,E')\right)\leq \frac{ra_{E'}}{m}$$ and by Proposition \ref{propequaldiscrep} we obtain that

$$r\left(\nu(S,E')-\nu(S_m,E')\right)\leq \frac{ra_{E'}}{m}=\frac{k+\ord_x(\Jac(f))}{m}\leq\frac{(k+1)d}{m}$$ where $a_{E'}$ is the log-discrepancy of $E'$, since $\ord_x(\Jac(f))\leq(k+1)(d-1)$.

\end{proof}

The following result is an immediate consequence of Proposition \ref{control1}

\begin{cor}\label{enoughanalyticsing}
Let $X\subset\P^k$ be an irreducible variety. Then,

$$\nu(S,X)=0\Leftrightarrow\nu(S_m,X)=0$$ for $m\gg1$. Moreover,

$$\lim_{n\to+\infty}\sup_{x\in\P^k}\nu\left(d^{-n}(f^n)^*S,x\right)=0\Leftrightarrow\lim_{n\to+\infty}\sup_{x\in\P^k}\nu\left(d^{-n}(f^n)^*S_m,x\right)=0$$ for $m\gg1$.
\end{cor}\Mp{enoughanalyticsing}

We would like to refine our approximation and replace $S_m$ by a current of integration on a hypersurface. Observe that for every finite collection of holomorphic germs $\sigma_1,\ldots,\sigma_N\in\OO_{\C^k,0}$, we can find a (Zariski) generic $\theta=(\theta_{i})\in\C^{N}$ such that, if we set $\sigma_{\theta}:=\sum_{i=1}^N\theta_{i}\sigma_i$ then 

$$\ord_0(\sigma_{\theta})=\min_{i=1,\ldots,N}\ord_0(\sigma_i).$$ 

In particular, fixing an orthonormal basis $\{\sigma_{m,j}\}$ of $\HH_m$ as before and given $\theta_m=(\theta_{m,j})\in\C^{N_m}$, we denote by $\varphi_{m,\theta}$ the function

\begin{equation}\label{varphisubtheta}
\P^k\ni x\mapsto\varphi_{m,\theta}(x):=\frac1{2m}\log\left(h^{\otimes(m+1)}\left(\sum_{j=1}^{N_m}\theta_{m,j}\sigma_{m,j}\right)(x)\right).
\end{equation}\Mpe{varphisubtheta} (Note the difference with $\varphi_m$ defined in \eqref{varphisubm}). It follows immediately that

$$\nu(\varphi_{m,\theta},x)\geq\nu(\varphi_m,x),\quad\forall\,\theta\in\C^{N_m},\,\forall\,x\in\P^k.$$ On the other hand, for each $x\in\P^k$ we can find a Zariski open set $V_x\subset\C^{N_m}$ such that

\begin{equation}\label{generictheta}
\nu(\varphi_{m,\theta},x)=\nu(\varphi_m,x)=\min_{j=1,\ldots,N_m}\ord_x(\sigma_{m,j}),\quad\forall\,\theta_m\in V_x.
\end{equation}\Mpe{generictheta}

We prove the following

\begin{lem}\label{picktheta}
Let $\EE$ be a finite family of irreducible subsets of $\P^k$. Then, there exists a Zariski open subset $U_m\subset\C^{N_m}$ such that
$$\ord_X(\varphi_{m,\theta})=\ord_X(\varphi_m)\quad\forall\,X\in\EE$$ for all $\theta\in U$.
\end{lem}\Mp{picktheta}

\begin{proof}
For all $j=1,\ldots,N_m$ and all $X\in\EE$ there exists a Zariski open subset $U_{X,j}\subset\P^k$ such that

$$\ord_X(\sigma_j)=\min_{z\in X}\ord_z(\sigma_j)=\ord_x(\sigma_j)\quad\forall\,x\in X\cap U_{X,j}.$$ Then it follows that for every $x$ in the Zariski open subset $U_{\EE}:=\bigcap_{X\in\EE,\,j}U_{X,j}\subset\P^k$ we have

$$\ord_x(\sigma_j)=\ord_X(\sigma_j)\quad\forall\,x\in U_{\EE}\cap X,\,j=1,\ldots,N_m.$$

Fixing $x\in U_{\EE}$, we now pick $V_x\subset\C^{N_m}$ as in \eqref{picktheta} and the conclusion follows.

\end{proof}

Using the same notation as above, we define the closed (1,1)-current

$$S_{m,\theta}:=\omega+dd^c\varphi_{m,\theta}$$ on $\P^k$. It also satisfies $S_{m,\theta}\geq-\frac1{m}\omega$ and we note that

$$S_{m,\theta}+\frac1{m}\omega=\frac1{m}[H_m]$$ where $[H_m]$ is the current of integration over the hypersurface given by

$$H_m=\text{div}\left(\sum_{j=1}^{N_m}\theta_{m,j}\sigma_{m,j}\right).$$ Carrying out the same argument given in Proposition \ref{control1}, we have proved the following crucial result

\begin{thm}\label{enoughhypersurface}
Let $\EE$ be any finite collection of irreducible algebraic varieties, let $f:\P^k\to\P^k$ be a holomorphic map of degree $d\geq2$, and let $S$ be a positive closed (1,1)-current on $\P^k$. Then, for every $m\geq1$, there exists a hypersurface $H_m\subset\P^k$ with

$$\ord_X(H_m)\leq\nu(S,X)\quad\forall\,X\in\EE,$$ and

$$\nu\left(d^{-n}(f^n)^*S,x\right)\leq d^{-n}\ord_x\left((f^n)^*H_m\right)+\frac{k+1}{m}$$ for all $x\in\P^k$.

\end{thm}\Mp{enoughhypersurface}

As an immediate consequence, if $S$ is a positive closed (1,1)-current and 

$$\sup_{x\in\P^k}d^{-n}\ord_x\left((f^n)^*H_m\right)\to0$$ for every $m\geq1$, we obtain the equidistribution of $S$ towards the Green current associated to $f$.


\section{Equidistribution}

In this section we prove our main results. Recall that if $X\in\EE_f$ then there exists $s\geq1$ (minimal) such that $f^{-s}(X)=X$. By Lemma \ref{wlog} we can assume without loss of generality that $s=1$.

\subsection{Proof of Theorems A and C}

By Theorem \ref{enoughhypersurface}, given any positive closed (1,1)-current and any finite family $\EE$ of irreducible varieties, we can find a sequence of hypersurfaces $H_m$ with the properties

$$\ord_X(H_m)\leq\nu(S,X)\quad\forall\,X\in\EE,$$ and

\begin{equation}\label{controlbyhyper}
0\leq\nu\left(d^{-n}(f^n)^*S,x\right)\leq d^{-n}\ord_x\left((f^n)^*H_m\right)+\frac{k+1}{m}
\end{equation}\Mpe{controlbyhyper} for all $x\in\P^k$. In particular, if the generic Lelong number $\nu(S,X)$ of $S$ along $X$ is zero for all $X\in\EE$, we have that $\ord_X(H_m)=0$ for all $X\in\EE$ and all $m\in\N$.\\

We recall from the introduction, Guedj's characterization of equidistribution:

\begin{equation}\label{gc}
d^{-n}(f^n)^*S\to T_f\Longleftrightarrow \lim_{n\to+\infty}\sup_{x\in\P^k}d^{-n}\ord_x\left((f^n)^*H_m\right)=0.
\end{equation}\Mpe{gc}

\begin{proof}[Proof of Theorem C]

Let $\EE_\DS$ be the collection of irreducible components of the totally invariant algebraic set constructed by Dinh-Sibony (see the introduction). By \cite{MR2468484}, Theorem 7.1 we have 

$$\ord_X(H_m)=0\quad\text{for all } X\in\EE_\DS\Rightarrow d^{-n}(f^n)^*H_m\rightarrow T_f$$ which, by the right implication in \eqref{gc}, gives us

$$\lim_{n\to+\infty}\sup_{x\in\P^k}d^{-n}\ord_x\left((f^n)^*H_m\right)=0$$ which implies

$$\limsup_{n\to+\infty}\sup_{x\in\P^k}\nu\left(d^{-n}(f^n)^*S,x\right)\leq\frac{k+1}{m}$$ for all $m\in\N$ by the inequality \eqref{controlbyhyper}. Letting $m\to+\infty$ we get

$$\lim_{n\to+\infty}\sup_{x\in\P^k}\nu\left(d^{-n}(f^n)^*S,x\right)=0$$ and by the left implication in \eqref{gc}, we conclude that

$$d^{-n}(f^n)^*S\to T_f.$$
\end{proof}

To prove Theorem A, we use the same arguments as above and reduce the problem to the case of $S$ the current of integration over some hypersurface $H\subset\P^k$. Here, the family $\EE_f$ is the exceptional family defined in Section 4.3.

\begin{proof}[Proof of Theorem A]
Let $H\subset\P^k$ be a hypersurface such that $\ord_X(H)=0$ for all $X\in\EE_f$, where $\EE_f$ is the exceptional family. Observe that $\ord_X(H)=0$ implies that $H|_X$ is a well defined Cartier divisor in $X$.\\

Let $X\in\EE_f$ and define $g:=f|_X:X\to X$; then

\begin{equation}\label{restriction}
d^{-n}\ord_x((f^n)^*H)\leq d^{-n}\ord_x((g^n)^*H|_X) \quad\forall\,x\in X.
\end{equation}\Mpe{restriction}

Let $\pi:\tilde X\to X$ be the normalization of $X$. By assumption, $\tilde X$ has at worst isolated quotient singularities, hence it is $\Q$-factorial and klt (see Section 3). Moreover, there exists a regular map $\tilde g:\tilde X\to\tilde X$ such that the diagram

\begin{equation*}
\xymatrix{\ar @{} [dr] 
\tilde X\ar[r]^{\tilde g}\ar[d]_{\pi}&\tilde X\ar[d]^{\pi}\\
X\ar[r]_{g}&X
}\end{equation*} commutes. Note that for every $y\in\tilde X$ and every germ $\phi\in\OO_{\tilde X,\pi(x)}$ if follows that $\ord_{\pi(y)}(\phi)\leq\ord_y(\phi\circ\pi)$. This, in particular implies that

$$\ord_{\pi(y)}((g^n)^*H|_X)\leq\ord_y(\pi^*(g^n)^*H|_X)=\ord_y(({\tilde g}^n)^*(\pi^*H|_X))$$ giving us

$$\sup_{x\in X}d^{-n}\ord_x((g^n)^*H|_X)\leq \sup_{y\in{\tilde X}}d^{-n}\ord_y(({\tilde g}^n)^*(\pi^*H|_X)).$$

By Theorem \ref{cocycleineq1} we have that

$$d^{-n}\ord_y(({\tilde g}^n)^*(\pi^*H|_X))\leq d^{-n}\mu_n^{\tilde X}(y)\ord_{{\tilde g}^n(y)}(\pi^*H|_X)$$ where $\mu_n^{\tilde X}$ is the submultiplicative cocycle defined by 

$$\mu_n^{\tilde X}(y)=C_{\tilde X}\left(\dim(\tilde X)+\ord_y(\CC_{{\tilde g}^n})\right)\quad\text{(see Section 4).}$$  

We know from Theorem \ref{asymptotic1} that there exist constants $C>0$ and $0\leq\rho<d$ such that

$$\sup_{y\notin E_{\tilde X}}\mu_n^{\tilde X}(y)\leq C\rho^n$$ for all $n\in \N$, where $E_{\tilde X}$ is the totally invariant algebraic set

$$E_{\tilde X}=\left\{x\in\tilde X\mid\mu^{\tilde X}_{\infty}(x)=d\right\}.$$

Recalling that the algebraic set $E_X:=\pi(E_{\tilde X})$ is totally invariant by $g$ (Corollary \ref{imageexceptional}), we hence obtain

\begin{multline}\label{finalineq}
\sup_{x\in X}d^{-n}\ord_x((f^n)^*H)\leq\sup_{x\notin E_X}d^{-n}\ord_x((f^n)^*H)+\sup_{x\in E_X}d^{-n}\ord_x((f^n)^*H)\leq\\ \leq  \sup_{x\notin E_{\tilde X}}d^{-n}\ord_y(({\tilde g}^n)^*(\pi^*H|_X))+\sup_{x\in E_X}d^{-n}\ord_x((f^n)^*H)\leq\\\leq C\left(\frac{\rho}{d}\right)^n+\sup_{x\in E_X}d^{-n}\ord_x((f^n)^*H).
\end{multline}\Mpe{finalineq}

We now proceed by induction on the partially ordered set $\EE_f$: If $X$ is a leaf (i.e. $E_X=\emptyset$) we obtain that 

$$\sup_{x\in X}\nu(d^{-n}(f^n)^*H,x)\leq C''(\rho/d)^n\to0$$ as $n\to+\infty$. In general, for $X\in\EE_f$ assume that for every $X'\preceq X$ we have that

$$\sup_{x\in X'}\nu(d^{-n}(f^n)^*H,x)\to0$$ as $n\to+\infty$.\\

Since every irreducible component $X'$ of $E_X$ satisfies $X'\preceq X$, we get

$$\sup_{x\in E_X}d^{-n}\ord_x((f^n)^*H)\to0$$ as $n\to+\infty$, implying

$$\sup_{x\in X}\nu(d^{-n}(f^n)^*H,x)\leq C'\left(\frac{\rho}{d}\right)^n+\sup_{x\in E_X}d^{-n}\ord_x((f^n)^*H)\to0$$ by inequalities \eqref{restriction} and \eqref{finalineq}. The desired conclusion then follows.

\end{proof}


\subsection{Proof of Corollary B}

If $f:\P^3\to\P^3$ is a holomorphic map of degree $d\geq2$ and $X\subset\P^3$ is an irreducible algebraic subset such that $f^{-1}(X)=X$, Corollary B would follow immediately if the normalization of every such $X$ has at worst (isolated) quotient singularities. Note that if $X$ is a curve, its normalization is automatically smooth.\\

Now, let $X\subset\P^3$ be a surface such that $f^{-1}(X)=X$ and let $g:=f|_X:X\to X$, $\tilde\pi:\tilde X\to X$ its normalization and $\tilde g:\tilde X\to\tilde X$ its holomorphic lift.\\

By \cite[Theorem B]{FavreSelfMapsSurfaces} or \cite{MR1044058} for every $x\in \tilde X$, we have

\begin{itemize}
\item[(i)] If $(\tilde X,x)$ is klt, then $(\tilde X,\tilde g(x))$ is klt.

If in addition we have that $\tilde g(x)=x$, then
\item[(ii)] if $x\in\CC_{\tilde g}$, then $(\tilde X,x)$ is klt;
\item[(iii)] if $x\notin\CC_{\tilde g}$, then $(\tilde X,x)$ is not klt (the singularity is log-canonical instead).\\
\end{itemize}

Observe that for any given $x\in\tilde X$, if the set $\bigcup_{n\geq1}\tilde g^{-n}(x)$ is not finite, then $x$ must be the image of some smooth point, and by (i) above $x$ is klt. Therefore, we just need to deal with the case $x$ totally invariant (hence fixed) by some iterate of $\tilde g$ reducing the problem to the case (iii).\\

The case (iii) can be divided in two subcases: Either $(\tilde X,x)$ is a cusp or not. The case $(\tilde X,x)$ a cusp can be ruled out by Theorem 1.4 in \cite{NakayamaRIMS}. If $(\tilde X,x)$ is not a cusp, Proposition 2.1 in \cite{FavreSelfMapsSurfaces} implies that we can find a proper modification $\bar\pi:\bar X\to\tilde X$ such that $\bar X$ has only klt singularities and $\tilde g$ lifts to a holomorphic map $\bar g:\bar X\to\bar X$ giving us the commutative diagram

\begin{equation}\label{lift3}
\xymatrix{\ar @{} [dr] |{\circlearrowright}
\bar X \ar[d]_{\pi} \ar[r]^{\bar g} & \bar X \ar[d]^{\pi}\\
X \ar[r]_{g} & X }\end{equation} where $\pi:=\tilde\pi\circ\bar\pi$. We can now apply Theorem A to \eqref{lift3} finishing the argument of Corollary B.\\

Alternatively, in the same setting as above,  in \cite{Zhang-classification} D.Q. Zhang found a concrete classification for $X\subset\P^3$. More precisely, Zhang states that either $\deg(X)=1$ (i.e.\ $X$ is a plane) or $X$ is a cubic given by one of the following four defining equations

\begin{itemize}
\item[(i)] $X_3^3+X_0X_1X_2$;
\item[(ii)] $X_0^2X_3+X_0X_1^2+X_2^3$;
\item[(iii)] $X_0^2X_2+X_1^2X_3$;
\item[(iv)] $X_0X_1X_2+X_0^2X_3+X_1^3$.
\end{itemize}

The surfaces given by (i) and (ii) are both normal with klt singularities. The singular locus of the varieties given by (iii) and (iv) is a single line which is totally invariant and their normalizations correspond to the smooth surface given by the one-point blowup of $\P^2$.

\subsection{Proof of Corollary D}

We provide a direct argument for this, not relying on the results given by Dinh-Sibony in \cite{MR2468484}.\\

As before, we define the Jacobian cocycle

$$\mu_n(x):=k+\ord_x(\Jac(f^n))$$ as described in Section 4.\\

The totally invariant set $$E:=\bigcup_{\delta>0}\bigcap_{n\in\N}\left\{x\in\P^k\mid \mu_n(x)\geq\delta d^n\right\}$$ is algebraic by Theorem \ref{algebraictotinv} and we know from Theorem \ref{asymptotic1} that there exist positive constants $C$ and $\rho<d$ such that

\begin{equation}\label{controlfinal}
\sup_{x\in\P^k\setminus E}\mu_n(x)\leq C\rho^n,\quad\forall\,n\in\N.
\end{equation}

Now the conclusion follows since

$$\sup_{x\in\P^k}\nu(d^{-n}(f^n)^*S,x)\leq\sup_{x\in\P^k\setminus E}\nu(d^{-n}(f^n)^*S,x)+\underbrace{\sup_{x\in E}\nu(d^{-n}(f^n)^*S,x)}_{=\,0}$$ and 

$$\sup_{x\in\P^k\setminus E}\nu(d^{-n}(f^n)^*S,x)\leq \sup_{x\in\P^k\setminus E}d^{-n}\mu_n(x)\leq C\left(\frac{\rho}{d}\right)^n\to 0$$ by Theorem \ref{cocycleineq1} and inequality \eqref{controlfinal}.

\def\lasp{\leavevmode\raise.45ex\hbox{$\lhook$}}
\providecommand{\bysame}{\leavevmode\hbox to3em{\hrulefill}\thinspace}
\providecommand{\MR}{\relax\ifhmode\unskip\space\fi MR }
\providecommand{\MRhref}[2]{%
  \href{http://www.ams.org/mathscinet-getitem?mr=#1}{#2}
}
\providecommand{\href}[2]{#2}

\end{document}